\documentclass[12pt]{article}
\usepackage{amsfonts,amsthm,amsmath,amssymb}
\usepackage{cite,microtype}

\usepackage[usenames]{color}
\usepackage[colorlinks=true,
linkcolor=webgreen,
filecolor=webbrown,
citecolor=webgreen]{hyperref}

\definecolor{webgreen}{rgb}{0,.5,0}
\definecolor{webbrown}{rgb}{.6,0,0}

\newtheorem{theorem}{Theorem}
\newtheorem{corollary}{Corollary}

\newcommand{\seqnum}[1]{\href{https://oeis.org/#1}{\rm \underline{#1}}}

\title{\Large\bf Egyptian Fractions with Denominators from\\ Sequences Closed Under Doubling}
\author{\large David Eppstein\\
Department of Computer Science\\
University of California, Irvine\\
Irvine, CA 92697\\
USA\\
\href{mailto:eppstein@uci.edu}{eppstein@uci.edu}}

\date{ }

\begin{document}
\maketitle  

\begin{abstract}
Resolving a conjecture of Zhi-Wei Sun, we prove that every rational number can be represented as a sum of distinct unit fractions whose denominators are practical numbers. The same method applies to allowed denominators that are closed under multiplication by two and include a multiple of every positive integer, including the odious numbers, evil numbers, Hardy--Ramanujan numbers, Jordan--P\'olya numbers, and fibbinary numbers.
\end{abstract}

\section{Introduction}

An \emph{Egyptian fraction} represents a rational number as a sum of distinct unit fractions; we allow $1/1$ as a unit fraction, but no larger integers. Several past works have studied Egyptian fractions with restricted denominators. In the mid-1950s, in connection with the still-open question of the termination of a greedy algorithm for finding representations with odd denominators, Breusch and Stewart showed that all rational numbers with odd denominators have Egyptian fractions with odd denominators~\cite{Bre-AMM-54,Ste-AJM-54}. One of Graham's first publications~\cite{Gra-PJM-64} Egyptian showed that Egyptian fractions with square denominators exist for all rationals in the intervals $[0,\pi^2/6-1)$ and $[1,\pi^2/6)$, and characterized Egyptian fractions with $k$th-power denominators for $k>2$. The last posthumous publication of Erd\H{o}s~\cite{ButErdGra-Int-15} showed that all natural numbers have Egyptian fraction representations whose denominators are the products of three prime numbers.

We find two properties of integer sequences that allow all rational numbers to be represented as Egyptian fractions with denominators from the sequence, or all rational numbers up to the natural limit of such representations, the sum of reciprocals of sequence elements.
\begin{itemize}
\item Sequence $S$ is \emph{closed under doubling} (or \emph{doubling-closed}) when, for all $x\in S$, $2x$ is also a member of $S$. That is, doubling a sequence element produces another sequence element.
\item Sequence $S$ is \emph{productive} when, for all $x\in\mathbb{Z}^+$, there exists $y\in\mathbb{Z}^+$ such that $xy$ is a member of $S$. That is, every integer has a multiple that is in the sequence.
\end{itemize}
Well-known sequences with both properties include the practical numbers, odious and evil numbers, Hardy--Ramanujan numbers, Jordan--P\'olya numbers, and fibbinary numbers (see \autoref{sec:sequences} for details on these sequences), allowing us to find Egyptian fractions using their elements as denominators. Our results positively resolve a conjecture of Sun according to which all rational numbers have Egyptian fractions with practical denominators~\cite{Sun-15}. Two other conjectures of Sun on Egyptian fractions whose denominators are the primes minus one or the primes plus one remain open; the sequences in those conjectures are not closed under doubling.

Our construction method multiplies the numerator of a fraction by a power of two, divides by the denominator, and separately finds the binary representations of the quotient and remainder. This method was used by Stewart to find Egyptian fractions with all denominators even~\cite{Ste-AJM-54} and called the ``binary remainder method'' in our earlier survey on Egyptian fraction construction algorithms~\cite{Epp-MER-95}.

\section{Doubling-closed and productive sequences}
\label{sec:sequences}

We briefly survey some notable integer sequences that are closed under doubling and productive. For these sequences, we also consider the convergence or divergence of the sum
\[ \sum_{x\in S} \frac{1}{x} \]
of reciprocals of elements of a sequence $S$, as this sum (when finite) forms a natural upper bound on Egyptian fractions with denominators in $S$. Some of these sequences are subsequences of others, but for representation by Egyptian fractions this is not redundant, because they will represent different ranges of rational numbers.

\begin{itemize}
\item The practical numbers (\seqnum{A005153}) are the numbers $n$ such that all integers $m\le n$ may be represented as sums of distinct divisors of $n$. For $n$ in this sequence, all rationals $m/n\in(0,1)$ have Egyptian fraction representations with denominators that are divisors of $n$. Several tables of Egyptian fraction expansions based on this principle were given by Fibonacci~\cite{Sig-FLA-02}, but the first explicit definition of the practical numbers was by Srinivasan in 1948~\cite{Sri-CS-48}.

A positive integer $n$ is practical if and only if, for each prime $p$ that divides $n$, the sum of divisors of the $(<p)$-smooth part of $n$ is at least $p-1$~\cite{Ste-AJM-54,Sie-AMPA-55}. This property is clearly preserved on multiplication by two, and can be obtained from any positive integer $n$ by multiplying it by a power of two bigger than all of its prime divisors. Therefore, the practical numbers are doubling-closed and productive. They have logarithmic asymptotic density~\cite{Wei-QJM-15}, like the prime numbers, so their sum of reciprocals diverges.

\item The odious numbers (\seqnum{A000069}) have binary representations with an odd number of nonzero bits, and the evil numbers (\seqnum{A001969}) have binary representations with an even number of nonzeros. Although named by Berlekamp, Conway, and Guy~\cite{BerConGuy-82}, their study goes back at least to Prouhet in 1851~\cite{Wri-AMM-59}. Doubling an odious or evil number shifts its binary representation without changing its parity. For any $n$ and sufficiently large $k$, $n(2^k+1)$ is evil, with a binary representation formed by two copies of the binary representation of $n$, separated by zeros, so the evil numbers are productive. The smallest multiplier whose product with $n$ is odious (\seqnum{A178757}) was proven to exist by Morgenbesser, Shallit, and Stoll~\cite{MorShaSto-JNT-11}. Because the odious and evil numbers have density $\tfrac12$ in the integers, the sums of their reciprocals diverge.

\item The Hardy--Ramanujan numbers (\seqnum{A025487}) have prime factorizations $2^{e_2} 3^{e_3} 5^{e_5}\cdots$ with $e_2\ge e_3\ge e_5\ge\cdots$. This ordering of exponents is preserved by doubling. A Hardy--Ramanujan multiple of any $n$ can be obtained from the prime factorization of $n$ by replacing each exponent with its maximum with all later exponents. Hardy and Ramanujan~\cite{HarRam-PLMS-16} proved that the set of these numbers below any given threshold $N$ has cardinality exponential in $\sqrt{\log N/\log\log N}$,  sufficiently sparse that their sum of reciprocals converges to a finite bound.

\item The Jordan--P\'olya numbers (\seqnum{A001013}) are the orders of automorphism groups of trees, studied by Jordan and P\'olya~\cite{Jor-Crelle-69,Pol-AM-37}, and are equivalently the products of factorials. Because 2 is a factorial, they are preserved by doubling. Each $n$ has a multiple in this sequence, for instance $n!$.  Their sum of reciprocals converges.

\item The fibbinary numbers (\seqnum{A003714}) have binary representations with no two consecutive nonzero bits. Doubling a fibbinary number shifts its binary representation without introducing new consecutive nonzeros. Multipliers for each $n$ showing that \seqnum{A003714} is productive, and an argument that these multipliers exist for all $n$, are given in \seqnum{A300867}. \seqnum{A003714} is so-named because the number of them with a given number of bits is a Fibonacci number. This implies that, up to any threshold $N$, there are $O(N^{\log_2\varphi})\approx N^{0.694242}$ fibbinary numbers, where $\varphi$ is the golden ratio, few enough that their sum of reciprocals converges.

\item The Moser--de Bruijn sequence (\seqnum{A000695}) consists of sums of distinct powers of four, and is not closed under doubling. Modifying it by including both the numbers of this form, and the doubles of numbers of this form, produces \seqnum{A126684}, which is closed under doubling. The existence in \seqnum{A126684} of a multiple of any $n$ follows from the same argument as the fibbinary numbers: the numbers of the form $(4^i-1)/3$, having binary representations that alternate between 0's and 1's, have by the pigeonhole principle two elements that are congruent modulo $n$, and the difference of these two elements is a multiple of $n$ in \seqnum{A126684}. \seqnum{A126684} grows quadratically, and is the fastest-growing sequence whose sums of pairs of elements include all positive integers. Its quadratic growth implies that its sum of reciprocals converges.

\item \seqnum{A116882} consists of the positive products $2^k\cdot\ell$ for which $2^k\ge\ell$, and is clearly closed under doubling. For every $n$, it contains the multiples $2^kn$ of $n$, for all $k$ large enough that $2^k\ge n$. When the numbers in \seqnum{A116882} up to some threshold $N$ are factored into a power of two and an odd part, there are $O(\log N)$ choices for the power of two and $O(\sqrt N)$ choices for the odd part, giving $O(\sqrt N\log N)$ total choices, few enough that the sum of reciprocals converges.
\end{itemize}

\section{Existence of representations}

\begin{theorem}
Let $S$ be a doubling-closed and productive set of positive integers, and let $q$ be a rational number with $0<q<\sum_{s\in S}1/s$.
Then $q$ has a representation as an Egyptian fraction with denominators in $S$.
\end{theorem}

\begin{proof}
Let $m=\min S$. Let $P$ be a prefix (that is, an initial subsequence) of the ascending sorted order of $S$, the longest prefix for which $q\ge\sum_{p\in P} 1/p$. Let $q'=q- \sum_{p\in P} 1/p$ be the remaining fraction to represent after choosing unit fractions with denominators in $P$. We find a representation $q'=x/y$ of $q$ as a fraction (not necessarily in lowest terms), where $y$ has a nontrivial odd factor and where $my\in S$. To do so, let $x'/y'$ be the lowest-terms representation of $q'$, multiply both $x'$ and $y'$ by three if $y'$ is a power of two, find a multiple $rmy'$ of $my'$ that belongs to $S$, and let $x=rx'$ and $y=ry'$.

Let $k=\lfloor\log_2 y\rfloor$, and divide $2^kmx$ by $y$ giving $2^kmx = ay + b$. Then $x/y<1/m$ (else we would have included more elements in $P$), and $2^k<y$, so $2^kmx<y^2$ and $a<y$. As with any remainder of division we have also $b<y$. Let $A=\{a_1,a_2,\dots\}$ and $B=\{b_1,b_2,\dots\}$ be sets of integers with $a=\sum_i 2^{a_i}$ and $b=\sum_j 2^{b_j}$, the sets of exponents in the binary representations of $a$ and $b$. By our choice of $k$, and because both $a$ and $b$ are less than $y$, each of the exponents in $A$ and $B$ is at most $k$.

Represent $q$ as a sum of unit fractions, by expanding
\begin{align*}
q &= \Bigl(\sum_{p\in P}\frac{1}{p}\Bigr) + \frac{x}{y} \\
&= \Bigl(\sum_{p\in P}\frac{1}{p}\Bigr) + \frac{2^kmx}{2^kmy} \\
&= \Bigl(\sum_{p\in P}\frac{1}{p}\Bigr) + \frac{ay+b}{2^kmy} \\
&= \Bigl(\sum_{p\in P}\frac{1}{p}\Bigr) +
   \Bigl(\sum_{a_i\in A} \frac{1}{2^{k-a_i}m}\Bigr) +
  \Bigl( \sum_{b_j\in B} \frac{1}{2^{k-b_i}my}\Bigr).
\end{align*}

All denominators in this representation are members of $P$, or can be obtained by repeatedly doubling $m$ or $my$, so all belong to $S$. Within each of the three summations of this representation, the unit fractions come from distinct elements of a set $P$, $A$, or $B$, and are therefore distinct from each other. The denominators of unit fractions coming from $P$ cannot coincide with any other denominators, because if they did we would have been able to include more elements in $P$. No denominators in the second summation can equal a denominator in the third summation, because of the nontrivial odd factor of $y$. Therefore, we have represented $q$ as a sum of distinct unit fractions all of whose denominators belong to $S$.
\end{proof}

\begin{corollary}
Every positive rational number has an Egyptian fraction representation with practical denominators, an Egyptian fraction representation with odious denominators, and an Egyptian fraction representation with evil denominators.
\end{corollary}

\begin{corollary}
For each of \seqnum{A025487}, \seqnum{A001013}, \seqnum{A003714}, \seqnum{A126684}, and \seqnum{A116882},
every positive rational number that is less than the sum of reciprocals of the sequence has an Egyptian fraction representation in which all denominators belong to the sequence.
\end{corollary}

\bibliographystyle{jis}
\bibliography{egypt}

\bigskip
\hrule
\bigskip

\noindent 
2010 \emph{Mathematics Subject Classification}:~Primary 11A67. 
Secondary 11A63, 11B83.

\medskip

\noindent 
\emph{Keywords}:~Egyptian fraction,
practical number,
odious number,
evil number,
Hardy--Ramanujan number,
Jordan--P\'olya number,
fibbinary number. 

\bigskip
\hrule
\bigskip

\noindent
(Concerned with sequences 
\seqnum{A000069},
\seqnum{A000695},
\seqnum{A001013},
\seqnum{A001969},
\seqnum{A003714},
\seqnum{A005153},
\seqnum{A025487},
\seqnum{A116882},
\seqnum{A126684},
\seqnum{A178757},
\seqnum{A300867}.)

\end{document}